\documentclass[leqno]{amsart}

\usepackage[english]{babel}
\usepackage{hyperref}
\usepackage{srcltx}
\usepackage{color}
\usepackage{amsmath,amscd,amssymb,amsthm,amsfonts}
\usepackage{graphicx}
\usepackage{mathrsfs}

\DeclareFontFamily{OT1}{pzc}{}
\DeclareFontShape{OT1}{pzc}{m}{it}{<-> s * [1.10] pzcmi7t}{}
\DeclareMathAlphabet{\mathpzc}{OT1}{pzc}{m}{it}

\usepackage{calrsfs}
\DeclareMathAlphabet{\pazocal}{OMS}{zplm}{m}{n}

\newtheorem{defin}{Definition}[section]
\newtheorem{theorem}{Theorem}[section]
\newtheorem{proposition}{Proposition}[section]
\newtheorem{lemma}{Lemma}[section]

\newtheorem*{corol}{Corollary}

\newcommand{\R}{\mathbb{R}}
\newcommand{\T}{\mathbb{T}}
\newcommand{\N}{\mathbb{N}}

\newcommand{\X}{\times}

\newcommand{\D}{\partial}

\newcommand{\curl}{\mathop\mathrm{curl}\nolimits}

\newcommand{\ve}{\mathbf}
\newcommand{\eqdef}{:=}

\numberwithin{equation}{section}

 \email{luca.bisconti@unifi.it}
 \email{davide.catania@unibs.it}

\begin{document}

\sloppy


\title[Horizontally filtered Boussinesq equation on a strip-like
region]{Global well-posedness of the two-dimensional horizontally
  filtered simplified Bardina turbulence model On a strip-like region}

\author[Luca Bisconti, Davide Catania]{}


\begin{abstract}
We consider the 2D simplified Bardina turbulence model, with
	horizontal filtering, in an unbounded strip-like domain. We prove
	global existence and uniqueness of weak solutions in a suitable
	class of anisotropic weighted Sobolev spaces.
\end{abstract}

\maketitle

\centerline{\scshape Luca Bisconti } \medskip {\footnotesize
  \centerline{Universit\`a degli Studi di Firenze}
  \centerline{Dipartimento di Matematica e Informatica ``U. Dini''}
  \centerline{Via S. Marta 3, I-50139 Firenze, Italia}
} 

\medskip

\centerline{\scshape Davide Catania} \medskip {\footnotesize
  \centerline{Universit\`a eCampus} \centerline{and}
  \centerline{Universit\`a degli Studi di Brescia} \centerline{Sezione
    Matematica (DICATAM)} \centerline{Via Valotti 9, I-25133 Brescia,
    Italia} }

\bigskip

\section{Introduction}
In the present paper we give some results mainly connected with the
regularity and the long-time behavior of the viscous simplified
Bardina turbulence model (with horizontal filtering) in a strip-like
region $\Omega\subseteq \R^2$, aimed at proving existence and
uniqueness of weak solutions in a suitable class of weighted Sobolev
spaces.

The Bardina closure model for turbulence was introduced {in 1980 by
  J. Bardina, J. H. Ferziger and W. C. Reynolds in \cite{Bardina}},
and later simplified and analyzed in \cite{LayLew} and in
\cite{CaoLunTiti}.  Indeed, the 3D simplified Bardina turbulence
system was proposed in \cite{LayLew} as a regularization model, for
small values of the scale parameter $\alpha$, of the 3D Navier--Stokes
equations for the purpose of numerical simulations. {Analysis of the
  global behavior of the pertinent solutions in a bounded domain, with
  periodic boundary conditions, appears in \cite{CaoLunTiti}.}
{Global} well-posedness for the 2D simplified Bardina model was
established in \cite{CaoTiti}.  {Again, in space-periodic domains, the
  inviscid simplified Bardina model is a regularizing system for the
  3D Euler equations; this because it is globally well-posed and it
  approximates the 3D Euler equations without adding spurious
  regularizing terms (see \cite{LayLew}).}

The behavior of the solutions for the simplified Bardina model {(in
  both 2D and 3D cases)} changes considerably {depending on whether
  the integration domain is bounded.  This is a basic point in
  studying general properties as regularity on the long-time period
  and dynamics (in particular, existence of attractors).}  {More
  generally, this remark applies to the solutions of a broader class
  of dissipative systems (see, e.g., \cite{AntZelik, EfZe,
    Miranv-Zelik, ZelikNSE}).}  In fact, unlike the case of bounded
domains, for some types of solutions to PDEs in unbounded regions
(such as spatially periodic patterns, travelling waves, etc...), {we
  can not expect to have uniform control on the energy; rather the
  energy of these solutions may blow up to infinity.} Again, due to
the unboundedness of $\Omega$, compactness for the {semigroup solution
  operator} can not be retrieved by using standard Sobolev embeddings
(there are no compact inclusions).  Hence, in this case, the
standard choices for bounded domains of the phase space, as
$L^p(\Omega)$, $W^{k,p}(\Omega)$ or $H^p(\Omega)$, $1 \leq p <
+\infty$, $k\in \N$, do not seem appropriate.

Even in the promising situation in which the solutions are bounded as
$|x|\to +\infty$ in $\Omega $, i.e. they belong to $L^\infty(\Omega)$,
the study of their behavior is not necessarily simplified since this
space is analytically awkward to use: on one hand, strong requirements
on the initial data are needed to have solutions in such a space; on
the other hand, the study of dynamics in this phase space {results} to
be more intricate since one does not have at disposal analytical
semigroups, maximal regularity properties for semigroups, etc...

A reasonable alternative is {using weighted Sobolev spaces (see, e.g.,
  \cite{Ab, BabinVish, Miranv-Zelik})} that, in principle, can contain
sufficiently regular, spatially bounded solutions on the long-time
period. In such a situation it is possible to study the semigroup
generated by the considered system and to check whether it admits a
global attractor in a {suitable} weighted phase space.  A main
advantage of this approach is that weighted Sobolev spaces are rather
handy to use since {they enjoy regularity, interpolation and embedding
  properties which are similar to those of the usual Sobolev spaces
  $W^{k,p}(\Omega)$ for bounded domain.}

However, proving estimates in such spaces is more complicated than in
the standard ones and, for our analysis, we find convenient to follow
the same path as in \cite{Babin-1992} (see also
\cite{Celebi-Kalant-Polat, Polat}).  In so doing, we consider the 2D
Navier--Stokes system in terms of {a stream function, $v$, and derive
  formally the 2D simplified Bardina with horizontal filtering.}

We now introduce the considered 2D simplified Bardina model for the
potential $v$ connected with the vector field $\mathbf{v}= (v_1,
  v_2)$ (here, $\mathbf{v}$ is a regularizing vector field associated
with the velocity field, $\mathbf{u}$, of the 2D Navier--Stokes
equations \eqref{eq:NS-2D-a} below, i.e. $\mathbf{v}\approx
\mathbf{u}$ and $v_1 = \D_2 v$, $v_2 = -\D_1 v$), on the strip-like
region $\Omega\subset \R^2$, i.e.:
{\begin{equation} \label{eq:Bardina-visc} \left. \begin{array}{ll}
        (1-\alpha^2\D_1^2) \Delta \D_t v +B(v, v) - \nu
        (1-\alpha^2\D_1^2)\Delta^2 v = g, & x \in \Omega \textrm{ and } t \in \R^+,\\
        v\vert_{t=0} = v_0(x), & x=(x_1, x_2) \in \Omega,\\
        v(x, t)= 0,\; \nabla v(x, t) =0, \; \D_1\nabla v(x, t) =0, &
        x\in\D \Omega \textrm{ and } t \in \R^+,
      \end{array}\right.
  \end{equation}
  where $B(v,v)\eqdef \D_2 v \D_1 \Delta v - \D_1 v \D_2 \Delta v$,}
$\alpha>0$ is a scale parameter, $\nu > 0$ is the kinematic viscosity,
$g$ is a forcing term, and the domain {$\Omega$} is defined by the
following inequalities (see \cite{Babin-1992, Celebi-Kalant-Polat,
  Polat}):
\begin{equation} \label{eq:channel-dom} b_1(x_1)\leq x_2 \leq b_2
  (x_1),\,\, x_1\in \R,
\end{equation}
where $b_1$ and $b_2$ are twice continuously differentiable functions
bounded over the entire $x_1$-axis according to
\begin{gather*}
  -M \leq b_1(x_1)\leq  x_2 \leq  b_2(x_1) \leq M,\,\, x_1\in \R,\\
  | b'_i(x_1) + b_i(x_1) |\leq c, \,\, i=1,2.
\end{gather*}

In order to formally derive system \eqref{eq:Bardina-visc}, we
consider the 2D Navier--Stokes equations in the space periodic setting
$\Omega=\T^2$ (although it would be sufficient to consider
  periodicity only in the $x_1$-direction), i.e.
{\begin{equation} \label{eq:NS-2D-a} \left.\begin{array}{ll}
        \partial_t \mathbf{u}+\nabla\cdot (\mathbf{u}\otimes \mathbf{u}) -
        \nu\Delta \mathbf{u} +\nabla\pi= \mathbf{f}(x,t)  \, , & x\in \Omega \textrm{ and } t \in \R^+,\\
        \nabla\cdot \mathbf{u}= 0 \, , &  x\in \Omega \textrm{ and } t \in \R^+,\\
        \mathbf{u}\vert_{t=0} = \mathbf{u}_0\, , & x\in \Omega,\\ 
      \end{array} \right.
  \end{equation}
  where} $\mathbf{u} (x,t)=(u_1, u_2)$ is the velocity field,
$\pi(x,t)$ denotes the pressure, $\mathbf{f} (x,t) = (f_1, f_2)$ is
the external force, and $\nu > 0$ the kinematic viscosity.

First, we rewrite the Navier--Stokes equations \eqref{eq:NS-2D-a} in
terms of the vorticity $\xi := \curl{\mathbf{u}} := \partial_1
u_2-\partial_2 u_1\in \R$ and then we introduce the stream function
$\omega$ associated to the velocity field {$\mathbf{u}$, i.e. a scalar
  function $\omega\in\R$ such that $\ve u=\curl\omega = (\D_2\omega,
  -\D_1\omega)\in \R^2$} (notice that $\xi=-\Delta\omega$), to
get
\begin{equation} \label{eq:Bard-visc} \left. \begin{array}{l} \Delta
      \D_t \omega +B(\omega, \omega) - \nu \Delta^2 \omega = g,\\
      \omega\vert_{t=0} = \omega_0,
    \end{array}\right.
\end{equation}
where the bilinear operator $B$ is as above (i.e.  $B(\omega, \omega)=
\D_2 \omega \D_1 \Delta\omega - \D_1\omega \D_2\Delta \omega$) and $g=
\D_2f_1 -\D_1 f_2$.

For a function $w$, we introduce the horizontal filter (related to the
horizontal Helmholtz operator), given by
\begin{equation} \label{eq:h-op} A_h = I -\alpha^2\D_1^2, \,\,\,
  \textrm{and}\,\,\, \overline{w}^{{}_h} := A_h^{-1}w.
\end{equation}
As discussed in \cite{Ali, Germano, LayLew-a}, from the point of view
of the numerical simulations, this filter is less memory consuming
with respect to the standard one. Further, another interesting feature
of this filter is that, even in the case of domains which are not
periodic in the vertical direction, there is no need to introduce
artificial boundary conditions for the Helmholtz operator (see,
  e.g., \cite{Ali, Berselli, BerCat, BerCat-2,
    BisCat}). 

We set $v\eqdef\overline{\omega}^{{}_h}$ (and $(v_1,
v_2)=\mathbf{v}:=\overline{\mathbf{u}}^{{}_h}$, with $v_1 = \D_2 v$,
$v_2 = -\D_1 v$) and solve the interior closure problem by using the
approximations
\begin{equation*}
  \overline{B(\omega, \omega)}^{{}_h} \approx
  \overline{B(\overline{\omega}^{{}_h},
    \overline{\omega}^{{}_h})}^{{}_h} =:
  \overline{B(v, v)}^{{}_h},
\end{equation*}
to get the following initial value problem:
\begin{equation*}
  \left. \begin{array}{l}
      \Delta    \D_t  v + \overline{B(v, v)}^{{}_h} - \nu  \Delta^2 v =    \overline{g}^{{}_h},\\
      v \vert_{t=0} = \overline{\omega_0}^{{}_h}.
    \end{array}\right.
\end{equation*}
{By applying the operator $A_h= I -\alpha^2\D_1^2$ to the above system,}
term by term, and considering the {obtained equations} on the
channel-like domain described by \eqref{eq:channel-dom} {(introducing suitable boundary conditions), we get}
\eqref{eq:Bardina-visc}.  Here and in the sequel, for simplicity, we
{always} assume that $g(x,t)=g(x)$.

Set the following anisotropic Sobolev spaces:
\begin{align*}
  &H^{2,h}_0 =\{ f\in W^{1,2}(\Omega)\, \colon \, \D_1 \nabla f \in
  L^2(\Omega), \nabla\cdot f =0 \textrm{ and }
  f\vert_{\D\Omega}=0\},\\
  \intertext{and} &H^{3,h} =\{ f\in W^{2,2}(\Omega)\, \colon \, \D_1
  \Delta f \in
  L^2(\Omega)\}. 
\end{align*}
As first step in our analysis we provide an existence theorem to
\eqref{eq:Bardina-visc} in standard anisotropic Sobolev spaces. In
  this case we deal with a proper class of weak solutions to the
  considered problem (see Definition~\ref{def:weak-sol} below).  This
result reads as follows.
 
  \begin{theorem}\label{thm:existence-no-w}
    Let $v_0\in H^{3,h}\cap H^{2,h}_0$ and $g \in L^2$.  Then,
    there exists a unique weak solution $v$ of the problem
    \eqref{eq:Bardina-visc}.
  \end{theorem}

  In the main theorem of the paper we prove that the global weak
  solution $v=v(t)$ in Theorem~\ref{thm:existence-no-w} is actually
  defined in a suitable class of anisotropic weighted Sobolev spaces
  (see Theorem~\ref{thm:weight-existence} below).

  In proving Theorem~\ref{thm:weight-existence} we do not follow
  directly the scheme behind the standard Aubin--Lions lemma,
  rather we use a different compactness method (see
    \cite[Corollary~2.34]{Simon}, see also Lemma~\ref{utility-lemma}
    below) by which we perform our analysis on approximating open
    bounded subsets $\pazocal{O}$ of $\Omega$. This result allows us
    to surmount the difficulties due to the boundary conditions and
    the unboundedness of the considered domain $\Omega$.

  The results obtained here open the way to the analysis of
  dynamics in terms of attractors; this will be the matter of a
  forthcoming paper. \smallskip

  \noindent \textbf{Plan of the paper}. In
  Section~\ref{sec:preliminary} we introduce the main notation and we
  also give some preliminary results. In Section~\ref{sec:weak-sol} we
  give the precise definition of weak solution, we state our main
  result (Theorem~\ref{thm:weight-existence}) and we also present some
  remarks on the existence of weak solutions.
  {Section~\ref{sec:no-weight} is devoted to the proof of
    Theorem~\ref{thm:existence-no-w}.}  In
  Section~\ref{sec:weight-result}, we study
  problem~\eqref{eq:Bardina-visc} in suitable Sobolev weighted spaces
  proving Theorem~\ref{thm:weight-existence}. Finally, the
    appendix is dedicated to the properties of the weight functions
    used to define the weighted Sobolev spaces used in
    Theorem~\ref{thm:weight-existence}.

  \section{Notation and preliminary
      results} \label{sec:preliminary} In what follows, we denote by
  $L^p:=L^p(\Omega)$, and $W^{k,p}:=W^{k,p}(\Omega)$, with
  $H^{k}:=W^{2,k}$, $k,p\in \N$, the usual Lebesgue and Sobolev
  spaces, respectively.  Also, we denote by $(\, ,\, $) and $\| \cdot
  \|$ the standard $L^2$-inner product and norm in $L^2(\Omega)$,
  respectively. We denote by $(H^k)'$ the dual space to $H^{k}$,
  $k,p\in \N$, and this notation will be adapted in a straightforward
  manner, when it makes sense, to the further spaces that will be
  introduced in the sequel.

  Given a Banach space $X$, for $p\in [1, \infty)$, we denote the
  usual Bochner spaces $L^p(0, T; X)$ with associated norm
  $\|f\|^p_{L^p(0, T; X)}:=\int^T_0 \|f(s)\|^p_X ds$ (the lower bound
  of $\|f(s)\|_X$ if $p=\infty$), with $\|\cdot \|_X$ the norm of $X$.

  Hereafter, $C$ will denote a dimensionless constant which might
  depend on the shape of the domain $\Omega$ and that may assume
  different values, even in the same line.

  Let us introduce the following function spaces:
  \begin{align*}
    & H:= \{ f \in L^2(\Omega) \, \colon\, \nabla \cdot f =0 \textrm{
      and } f
    =0 \textrm{ on } \partial\Omega\}, \\
    & H^{1,h} := \{ f \in L^2(\Omega)  \, \colon\,  \D_1 f \in  L^2(\Omega) \}, \\
    & H^{2,h} := \{f \in H^{1,h} \, \colon\, \D_1 \nabla f \in
    L^2(\Omega) \} ,\\
    & H^{3,h} := \{f \in H^{2,h} \, \colon\, \D_1\Delta f \in
    L^2(\Omega) \},
  \end{align*}
  and $H^{l,h}_{0}:=H^{l,h}\cap H$, $l=1,2,3$.

  \subsection{Weighted Sobolev spaces: Basic properties and
      related inequalities} \label{ssec:weight} Here, we consider a
  family of functions $\varphi_\rho(x, \alpha_1, \alpha_2, \epsilon,
  \gamma)=\varphi(x, \alpha_1, \alpha_2, \epsilon, \rho, \gamma)$
  enjoying the following properties, analogue to those listed in
  {\cite[\S 2.2, (A), pg. 383 ]{Celebi-Kalant-Polat} (see also
    \cite{Babin-1992})}:
  \begin{align*}
    & \varphi\geq 1, \quad \varphi(x,\alpha_1,\alpha_2,\epsilon,\rho,\gamma)=\varphi(\epsilon x,\alpha_1,\alpha_2,1,\rho,1)^\gamma, \\
    & \varphi(x,\alpha_1,\alpha_2,1,\rho,\gamma) \text{ does not depend on } \rho \text{ if } |x|\leq\rho, \\
    & \varphi(x,\alpha_1,\alpha_2,1,\rho,\gamma) = \varphi(\rho +1,\alpha_1,\alpha_2,1,\rho,\gamma) \text{ as } |x|\geq\rho +1 \\
    & \varphi(x,\alpha_1,\alpha_2,\epsilon,\rho_1,\gamma)\geq \varphi(x,\alpha_1,\alpha_2,\epsilon,\rho_2,\gamma) \text{ for } \rho_1\geq\rho_2\geq 1, \gamma\geq 0, \\
    & \underset{\rho \to +\infty}{\lim} \varphi(x, \alpha_1, \alpha_2,
    1, \rho, \gamma) = (1 + |
    x_1|^{\alpha_1}+|x_2|^{\alpha_2})^{\frac{\gamma}{2}}=: \phi(x,
    \alpha_1, \alpha_2, 1, \gamma/2),
  \end{align*}
  where we have introduced the anisotropic weight function $\phi (x,
  \alpha_1, \alpha_2, \epsilon, \gamma) = (1 + |\epsilon
  x_1|^{\alpha_1}+|\epsilon x_2|^{\alpha_2})^{\gamma}$ (see, e.g.,
  \cite{Babin-1992, EfZe}), which is suitable for the anisotropic
  Sobolev spaces we will consider in the sequel (see the definition of
  $H^{l,h}_\gamma$, $l=1,2,3$, below).

  For the remainder of the paper we always assume $\alpha_1=3$,
  $\alpha_2=2$ and we use the compact notations $\phi:= \phi (x,
  \epsilon, \gamma) := \phi (x, 3, 2, \epsilon, \gamma)$ and $\varphi
  :=\varphi (x, \epsilon, \rho, \gamma): = \varphi (x, 3, 2, \epsilon,
  \rho, \gamma)$.

  Again, arguing as in \cite{Babin-1992, Celebi-Kalant-Polat}, we take
  $\psi := \varphi^{1/2}$. Notice that we can choose $\varphi$ so that
  \begin{equation} \label{eq:weight-reg} |\D^\beta \psi^2|\leq
    C\epsilon^{|\beta|}\psi^2 \textrm{ for every multi-index }
    \beta=(\beta_1,\beta_2),\,\, |\beta|\leq 3 \textrm{ and }
    \beta_2\leq 2.
  \end{equation}
  This property will play a crucial role in the subsequent
    computations.  \smallskip

  We denote by $H^l_\gamma:=H^l_\gamma(\Omega)$ the space of functions
  equipped with the following norm:
  \begin{equation*}
     \|v\|_{l,\gamma}^2:=\sum_{|\beta|\leq l}\|\D^\beta v\|_\gamma^2,
  \end{equation*}
  where
  \begin{equation*}
      \|v\|_\gamma^2:=\int_\Omega |v|^2 (1 + |x_1|^3 + |x_2|^2)^\gamma dx,
    \,\,\gamma >0\,\, \textrm{ and } \,\, \D^\beta :=\frac{\D^{\beta_1
        +\beta_2}}{\D x_1^{\beta_1}\D x_2^{\beta_2}}, \,\, (\beta_1, \beta_2)\in \N^2.
  \end{equation*}


  Using the above notation, we introduce the further spaces
  \begin{equation} \label{eq:aws-spaces}
    \begin{aligned}
      & H_\gamma := H_{0, \gamma} := \{ f\in H\, \colon\, \| f \|_\gamma <+\infty\},\\
      & H^{1,h}_\gamma := \{ f \in H \, \colon\, \|\D_1
      f\|_\gamma <+\infty \}, \\
      & H^{2,h}_\gamma := \{f \in H^{1,h}\cap H \, \colon\, \|\D_1 \nabla f \|_\gamma <+\infty \} ,\\
      & H^{3,h}_\gamma := \{f \in H^{2,h}\cap H \, \colon\,
      \|\D_1\Delta f \|_\gamma <+\infty \}.
    \end{aligned}
  \end{equation}

  Let us recall the following results taken from \cite{Babin-1992}
  (see also \cite{Celebi-Kalant-Polat}).
  \begin{proposition}
    If $v\in H_0^1(\Omega)$, then
    \begin{equation}
      \big| \| \psi \nabla v \| - \| \nabla (\psi v)\|\big|\leq C \epsilon
      \|\psi v\|.
    \end{equation}
    If $v \in H^2_0$, then it holds true that
    \begin{equation}
      \big| \| \psi \Delta v \| - \| \Delta (\psi v)\|\big|\leq C \epsilon
      \|\psi |v| + \psi |\nabla v| \|.
    \end{equation}
  \end{proposition}
  Next, we have a weighted version of the classical Poincar\'e
  inequality.
  \begin{proposition}
    Let $v\in H^1_0$ with $v\vert_{\D\Omega}=0$. Then it holds true
    that
    \begin{equation}
      \|\psi v\| \leq 2 \lambda_1^{-1}\|\psi \nabla v\|.
    \end{equation}
    Let $\epsilon$ in the definition of $\varphi$ be sufficiently
    small. Let $v\in H^2\cap H^1_0$. Then
    \begin{equation}
      \|\psi \nabla v\| \leq 2 \lambda_1^{-\frac{1}{2}}\|\psi \Delta v\|.
    \end{equation}
  \end{proposition}
  \begin{proposition}
    It holds true that
    \begin{align}
      & \|v\|_{1,\gamma}\leq \|\psi \nabla v\|\,\, \forall v\in
      H^1_\gamma,\,\, v\vert_{\D \Omega}=0,
      \intertext{and} & \|v\|_{2,\gamma}\leq \|\psi \Delta
      v\|\,\, \forall u\in H^2_\gamma,\,\, v\vert_{\D \Omega}=0.
    \end{align}
  \end{proposition}
  We also have the following controls in the $L^4$-norm.
  \begin{proposition} Let $v\in H^1_\gamma$, with
    $v\vert_{\D\Omega}=0$. Then
    \begin{equation} \label{eq:l4-pesata}
      \begin{aligned}
        \|\psi v \|_{L^4}\leq C\|\nabla (\psi v)\| &\leq C \|\psi
        \nabla v\| +
        C\|v\nabla \psi\|\\
        & \leq C \|\psi \nabla v\| + C\epsilon\|v \psi\|.
      \end{aligned}
    \end{equation}
    Further, if $v\in H^2_\gamma$, with $v\vert_{\D\Omega}=0$, then
    \begin{equation}
      \|\psi \nabla v \|_{L^4}\leq C\epsilon \|\psi \nabla v\| + C\|\psi
      \Delta v\|. 
    \end{equation}
  \end{proposition}

  \section{Weak solutions and existence results} \label{sec:weak-sol}
  Consider the simplified-Bardina model \eqref{eq:Bardina-visc}.
  Observe that the bilinear form
  \begin{equation*}
    B(u,v)\eqdef \, \D_2 v \D_1 \Delta u - \D_1 v \D_2 \Delta u =
    \D_1(\D_2v \Delta u)-\D_2(\D_1 v \Delta u)
  \end{equation*}
  is such that
  \begin{align*}
    \bigl(B(u,v),w\bigr) = & \int \D_1(\D_2v \Delta u)w-\D_2(\D_1 v
    \Delta u)w
    =\int \D_1 v \Delta u \D_2 w - \D_2 v \Delta u \D_1 w\\
    = & -\int \D_1 w \Delta u \D_2 v - \D_2 w \Delta u \D_1 v =  -\bigl(B(u,w),v\bigr),\\
    \intertext{and} \bigl(B(u,v),v\bigr) = \, & 0,
  \end{align*}
  where the second line is obtained integrating by parts and
  exploiting the boundary conditions.  Here and in the sequel,
    unless stated otherwise, we drop the $dx$ in the space-integrals
    to keep the notation as compact as possible.  \smallskip

  We now give the following definition.

  \begin{defin} \label{def:weak-sol} Given $v_0\in H^{3,h}\cap
      H^{2,h}_0$ and $g \in L^2(\Omega)$, we say that $v\in
    L^\infty_{\rm loc}(\R; H^{2,h}_0\cap H^{3,h})$ is a weak solution
    of \eqref{eq:Bardina-visc} if $v_t\in L^2_{\mathrm{loc}}(\R;
    H^{2,h}_0)$ and
    \begin{align*}
      & (\nabla v_t,\nabla\mathpzc{h})+\alpha^2(\D_1\nabla v_t,
      \D_1\nabla\mathpzc{h})
      +\nu(\Delta v, \Delta\mathpzc{h}) +\nu\alpha^2(\D_1\Delta v, \D_1\Delta\mathpzc{h}) \\
      & \qquad = \big( B(v,v), \mathpzc{h} \big)-(g,\mathpzc{h})
    \end{align*}
    for every $\mathpzc{h}\in H^{2,h}_0\cap H^{3,h}(\Omega)$, for
    a.e. $t\in \R$ (and the initial datum is assumed in weak sense).
  \end{defin}

  In the next section we give a proof of
  Theorem~\ref{thm:existence-no-w} that guarantees existence and
  uniqueness of a weak solution to problem~\eqref{eq:Bardina-visc}.

   The anisotropic weighted Sobolev spaces introduced in
    \eqref{eq:aws-spaces} provide the appropriate functional framework
    for studying the existence of weak solutions to
    \eqref{eq:Bardina-visc} enjoying extra regularity
    properties. Then, in Section~\ref{sec:weight-result} we prove our
    main result, that reads as follows.

  {\begin{theorem} \label{thm:weight-existence} Let $g \in H_{0,
      \gamma}$. Then, for any $v_0 \in
      H^{3,h}_\gamma\cap H^{2, h}_0$ and $T>\tau$ given, 
    the weak solution $v$ of \eqref{eq:Bardina-visc} provided by
    Theorem~\ref{thm:existence-no-w} is such that $v\in
    L^\infty(\tau, T; H^{2, h}_{ \gamma})\cap
    L^2(\tau, T; H^{3, h}_{ \gamma})\cap C(\tau, T;
    H^{1,h}_{ \gamma})$ and $v_t \in L^2(\tau, T;
    H^{1,h}_{ \gamma})$. 
  \end{theorem}}

{\begin{corol} Under the hypotheses of 
  Theorem~\ref{thm:weight-existence} it holds true that $v\in
    L^\infty_{\mathrm{loc}}(0,\infty; H^{2, h}_{ \gamma})\cap
    L^2_{\mathrm{loc}}(0,\infty; H^{3, h}_{ \gamma})\cap C(0, \infty;
    H^{1,h}_{ \gamma})$ and $v_t \in L^2_{\mathrm{loc}}(0,\infty;
    H^{1,h}_{ \gamma})$.
  \end{corol}
  }
  \section{Existence in anisotropic Sobolev spaces}
  \label{sec:no-weight}
  This section is devoted to the proof of
  Theorem~\ref{thm:existence-no-w}, which provides the existence of a
  unique weak solution of the problem~\eqref{eq:Bardina-visc}. Since
  the proof follows standard methods, we proceed formally in order to
  find appropriate a priori estimates. A rigorous proof can be easily
  obtained by introducing a Galerkin approximation and finding similar
  estimates.

  We are now ready to proceed with the proof of
  Theorem~\ref{thm:existence-no-w}.

\begin{proof}[Proof of Theorem~\ref{thm:existence-no-w}]
  Testing formally \eqref{eq:Bardina-visc} against $v$, we get
  \begin{gather}
    \frac{1}{2}\frac{d}{dt} \big(\|\nabla v\|^2 + \alpha^2\|\D_1\nabla
    v\|^2 \big) + \nu \big(\|\Delta v\|^2 + \alpha^2\|\D_1\Delta v\|^2
    \big) \leq |(g,v)|.
  \end{gather}
  Since $|(g,v)|\leq \lambda_1^{-1}\|g\| \cdot \|\Delta v\| \leq
  \frac{1}{2\nu\lambda_1^2}\|g\|^2 + \frac{\nu}{2}\|\Delta v\|^2$, we
  deduce
  \begin{align*}
    \frac{d}{dt} \big(\|\nabla v\|^2 + \alpha^2\|\D_1\nabla v\|^2
    \big) + \nu \big(\|\Delta v\|^2 + \alpha^2\|\D_1\Delta v\|^2 \big)
    \leq \frac{1}{\nu\lambda_1^2}\|g\|^2,
  \end{align*}
  which implies
  \begin{align*}
    & \|\nabla v(t)\|^2 + \alpha^2\|\D_1\nabla v(t)\|^2 + \nu \int_0^t
    \big(\|\Delta v(s)\|^2 + \alpha^2\|\D_1\Delta v(s)\|^2
    \big) ds \\
    & \qquad \leq \frac{1}{\nu\lambda_1^2}\|g\|^2 t + \|\nabla
    v(0)\|^2 + \alpha^2\|\D_1\nabla v(0)\|^2,
  \end{align*}
  so that $v\in L^\infty_{\rm loc}(0, \infty; H^{2,h}_0)\cap
    L^2_{\rm loc}(0, \infty; H^{2,h}_0 \cap H^{3,h})$.
		
  Multiplying \eqref{eq:Bardina-visc} against $v_t$ and
    integrating over $\Omega$, we get
  \begin{equation} 
    \begin{aligned} 
      \frac{\nu}{2} \frac{d}{dt} \big(\|\Delta v\|^2 +
      \alpha^2\|\D_1\Delta v\|^2 \big)    {+
      \|\nabla v_t\|^2}
 & {+ \alpha^2\|\D_1\nabla v_t\|^2 }  \\
      &\leq |\big(g, v_t\big)| + |\big( B(v, v), v_t
      \big)|. \label{eq:stima-ut}
    \end{aligned}
  \end{equation}
  We have $|(g,v_t)|\leq \lambda_1^{-1/2} \|g\| \cdot \| \nabla v_t \|
  \leq \frac{\|g\|^2}{\lambda_1}+\frac{\|\nabla v_t\|^2}{4}$ and,
  thanks to the H\"older, the Gagliardo--Nirenberg and the Young
  inequalities, we have also
  \begin{align*}
    | (B(v,v),v_t) | & \leq  \| \D_1 v\|_{L^\infty} \| \Delta v\| \, \| \D_2 v_t\| + \| \D_2 v\|_{L^4} \| \Delta v\| \, \| \D_1 v_t \|_{L^4} \\
    & \leq   \| \D_1 \Delta v\| \, \| \Delta v\| \, \| \nabla v_t\| + C \| \Delta v \|^2 \| \D_1\nabla v_t \| \\
    & \leq \frac{\|\nabla v_t\|^2}{4} + \frac{\alpha^2}{2}\| \D_1
    \nabla v_t \|^2 + C \nu(\| \Delta v\|^2+\alpha^2 \| \D_1
      \Delta v\|^2)\|\Delta v\|^2,
  \end{align*}
  for a suitable constant $C=C(\lambda_1, \alpha,
  \nu)>0$. {Plugging this estimate in \eqref{eq:stima-ut}}, we obtain
  \begin{align*}
    & \nu \frac{d}{dt} \big(\|\Delta v\|^2 + \alpha^2\|\D_1\Delta
    v\|^2 \big)   +   \|\nabla v_t\|^2 + \alpha^2\|\D_1\nabla v_t\|^2 \\
    &\qquad \leq C \|g\|^2 + C \nu(\| \Delta v\|^2+\alpha^2 \|
      \D_1 \Delta v\|^2)\|\Delta v\|^2.
  \end{align*}
  Since we have already proved 
  {that $v\in L^2_{\rm
    loc} (H^2)$,} an application of the Gr\"onwall lemma gives the
  claimed regularity of $v$ (here we use the full regularity of
  $v_0$), and consequently by the previous inequality, the
  regularity of $v_t$.

  Lastly, notice that the proof of the uniqueness of weak solutions is
  quite standard and very similar to the proof of uniqueness for the
  case of a bounded domain (mainly because of the validity of the
  Poincar\'e inequality) and this last part of the proof is left to
  the reader.
\end{proof}

\section{Weak solutions in anisotropic weighted Sobolev
  spaces} \label{sec:weight-result} In what follows we prove the main
result of the paper, i.e. Theorem~\ref{thm:weight-existence}.  Let us
consider the weigth function $\phi=(1 + |x_1|^3 +|x_2|^2)^\gamma$,
$\gamma>0,$ introduced in Subsection~\ref{ssec:weight}, and the
approximating function $\varphi = \varphi(x, \epsilon, \rho, \gamma)$
with $\psi=\varphi^{\frac{1}{2}}$.

We state the following technical lemma.

\begin{lemma}\label{w-funcs}
  Under the assumption $\gamma\leq 2/3$, setting $\psi =
  \varphi^{1/2}$, then it holds true that
  \begin{equation} \label{control-on-w-funcs} |\D^\beta \psi^2| \leq
    C\epsilon^{|\beta |} \psi,\,\, \qquad \beta=(\beta_1,\beta_2),\,\,
    0< |\beta|\leq 3,\,\, \beta_2 \leq 2.
  \end{equation}
\end{lemma}

The precise construction of the weight function $\varphi$ and the
proof of this lemma are postponed to Appendix~\ref{app:weight}.  For
the remainder of the paper we always assume that $\gamma \leq 2/3$.

\medskip

In the proof of existence in weighted spaces, we will use the
following result (see also \cite[Theorem 2.2]{Garrido} and
  \cite{Ahn}) to overcome the difficulties arising because of the
unboundedness of the strip-like region $\Omega$. 

\begin{lemma}[Corollary 2.34, \cite{Simon}] \label{utility-lemma} Let
  $\Theta$ be a bounded set of $\R^d$, $X \subset E$ Banach spaces
  with compact injection. Consider $1 \leq p < q \leq +\infty$.
  Suppose that $\pazocal{F} \subset L^p(\Theta; E)$ satisfies
  \begin{itemize}
  \item[(i)] $\forall \pazocal{W} \subset\subset \Theta,
    \underset{k\to 0}{\lim} \underset{f \in \pazocal{F}}{\sup}
    \|\tau_kf - f \|_{L^p(\pazocal{W}; E)} = 0$ (where $\tau_k f$ is
    the translation given by
    $\tau_k f (x) = f (x + k))$,\\[-1 em]
  \item[(ii)] $\pazocal{F}$ is bounded in $L^q(\Theta; E) \cap
    L^1(\Theta, X)$.
  \end{itemize}
  Then, $\pazocal{F}$ is precompact in $L^p(\Theta; E)$.
\end{lemma}

We are now ready to prove Theorem~\ref{thm:weight-existence}.

\begin{proof}[Proof of Theorem~\ref{thm:weight-existence}]
  Since $H^{3,h}_\gamma\cap H^{2,h}_0$ is separable and the set
  $\pazocal{V}= \{v\in C_0^\infty(\Omega) : \nabla \cdot v=0\}$ is
  dense in $H^{3,h}_\gamma\cap H^{2,h}_0$, there exists a
  sequence of linearly independent elements $\{w_1, w_2,\ldots\}
  \subset \pazocal{V}$ which is complete in $H^{3,h}_\gamma\cap
  H^2_0$.  Denote $H_m := \textrm{span} \{w_j\}_{j=1,\ldots, m}$ and
  consider the projector $P_m(v) = \sum^m_{j=1}(v, w_j)w_j$.  A
  function \[v^m= \sum_{j=1}^m
  a_j^m(t)w_j(x)\] 
  is an $m$-approximate solution of Equation~\eqref{eq:Bardina-visc}
  if
  \begin{align*}
    & (\nabla v_t^m,\nabla w_j)+\alpha^2(\D_1\nabla v_t^m, \D_1\nabla
    w_j)
    +\nu(\Delta v^m, \Delta w_j) +\nu\alpha^2(\D_1\Delta v^m, \D_1\Delta w_j) \\
    & \qquad = \big( B(v^m,v^m), w_j \big)-(g, w_j)
  \end{align*}
  for every $j=1,\ldots,m$. The existence of solutions is guaranteed
  by the Peano theorem.

  We split the proof in a number of steps.
  \begin{enumerate}
  \item We establish a priori estimates for $\{v^m\}$ in the space {$ L^\infty_{\mathrm{loc}}(0,\infty; H^{2, h}_{ \gamma})\cap
    L^2_{\mathrm{loc}}(0,\infty; H^{3, h}_{ \gamma})$.}
  \item We show that $\{v^m\}$ satisfies condition (ii) of
    Lemma~\ref{utility-lemma}: $\pazocal F :=
      \{v^m|_{\pazocal{O}}\}$ is bounded in
    $L^\infty(\tau,T;H^{2,h}_\gamma(\pazocal O))\cap
    L^1(\tau,T;H^{3,h}_\gamma(\pazocal O))$, where $\pazocal O$ is any
    open subset in $\Omega$ and $\Theta := (\tau, T)$.
  \item We show that $\{v^m\}$ satisfies condition (i) of
    Lemma~\ref{utility-lemma}:
    	
    $ \underset{k\to 0}{\lim}\, \underset{m\in N}{\sup} \|\tau_kv^m -
    v^m\|_{L^2(\tau ,T-k; H^{2,h}_{\gamma}(\pazocal{O}))} = 0.  $
  \item We apply Lemma~\ref{utility-lemma} and extract a subsequence
    still denoted by $v^m|_{\pazocal O}$ converging to some {$v$ in
    $L^2(\tau, T; H^{2, h}_\gamma(\pazocal{O}))$.} In particular, we observe that the relations in
      the previous two points are uniform with respect to $\pazocal
      O$.
  \item The limiting function $v$ is a weak solution.
  \item By interpolation, we obtain the time continuity of $v$ with
    values in $H^{1,h}_\gamma(\Omega)$.
  \end{enumerate}

  \smallskip

  \noindent \textbf{STEP 1:} \emph{Establishing a priori estimates in
    $ L^\infty_{\mathrm{loc}}(0,\infty; H^{2, h}_{ \gamma})\cap
    L^2_{\mathrm{loc}}(0,\infty; H^{3, h}_{ \gamma})$.}

  Here, we proceed again formally by dealing with $v$ and the
    equation satisfied by it.  However, the a priori estimates that we
    are about to derive can be rigorously justified. Indeed, a
    rigorous proof uses $v^m$ instead of $v$ and $w_j$ as test
    functions, as it will be made in the second and third steps
    below.
  
  We multiply equation~\eqref{eq:Bardina-visc} by $v\psi^2$ in
  $L^2(\Omega)$ and use integration by parts to get
  \begin{equation} \label{eq:stima} \allowdisplaybreaks
    \begin{aligned}
      & \frac{1}{2}\frac{d}{dt} \Big( \|\psi \nabla v\|^2 + \alpha^2
      \|\psi \D_1\nabla v\|^2 \Big) + \nu \| \psi \Delta v\|^2 +
      \nu\alpha^2\|\psi
      \D_1\Delta v\|^2 \\
      & = \big( B(v,v) -g, v\psi^2\big) -\int \nabla v_t \, v \nabla
      \psi^2 +\alpha^2\int \D_1^2\nabla
      v_t\,  v\nabla\psi^2\\
      &\quad - \alpha^2\int \D_1\nabla v_t \nabla v\D_1 \psi^2
      -\nu\int \Delta v \, v \Delta\psi^2
      - 2\nu\int \Delta v\nabla v \nabla\psi^2 \\
      & \quad -2\nu\alpha^2\int \D_1 \Delta v \D_1\nabla v \nabla
      \psi^2 - \nu\alpha^2\int \D_1\Delta v \D_1 v \Delta \psi^2
      -\nu\alpha^2\int \D_1\Delta v \Delta v  \D_1 \psi^2\\
      &\quad -2\nu\alpha^2\int \D_1\Delta v \nabla v \D_1\nabla\psi^2
      -\nu\alpha^2\int \D_1  \Delta v \, v  \D_1\Delta\psi^2 \\
      & = \big( B(v,v) -g, v \psi^2\big) +\sum_{j=1}^{10} L_j.
    \end{aligned}\!\!\!\!\!
  \end{equation}
  In particular, we have used that 
  \begin{align*}
    &-\int\D_1^2\Delta^2 v\, v \psi^2 = \int \D_1\Delta^2 v \D_1 v
    \psi^2 +
    \int \D_1\Delta^2 v v \D_1 \psi^2 \\
    =& -\int \D_1 \nabla\Delta v \D_1\nabla v \psi^2
    -\int \D_1\nabla\Delta v \D_1 v \nabla \psi^2 \\
    & \quad -\int \D_1\nabla\Delta v \nabla v \D_1 \psi^2
    -\int \D_1\nabla\Delta v v \D_1\nabla\psi^2 \\
    = & \int \D_1 \Delta v \D_1\Delta v \psi^2
    +2\int \D_1 \Delta v \D_1\nabla v \nabla \psi^2 + \int \D_1\Delta v \D_1 v \Delta \psi^2 \\
    & \quad +\int \D_1\Delta v \Delta v \D_1 \psi^2 +2\int \D_1\Delta
    v \nabla v \D_1\nabla\psi^2 +\int \D_1 \Delta v v \D_1\Delta\psi^2
  \end{align*}
  and noticed that all boundary terms are zero.

  Proceeding as in \cite{Celebi-Kalant-Polat}, we immediately have the
  existence of a constant $C=C(\lambda_1,\nu,\alpha)>0$ such that
  \begin{equation} \label{eq:nonlinear-t}
    \begin{aligned}
      \big(B(v,v), v\psi^2 \big) & =\int \big[\D_1 v \Delta v \D_2
      (v\psi^2)
      - \D_2 v \Delta v \D_1 (v\psi^2)\big]\\
      &=\int \big[\D_1 v \Delta v \, v\D_2\psi^2
      - \D_2 v \Delta v \, v\D_1 \psi^2\big]\\
      &\leq 2\|v\|_{L^\infty}\|\psi\Delta v\|\|\psi\nabla v\|\\
      &\leq \epsilon \nu \|\psi\Delta v\|^2 + C\epsilon
      \|\psi\nabla v\|^2
    \end{aligned}
  \end{equation}
  and
  \begin{align*}
    \big|\big(g, v\psi^2 \big)\big|& \leq C \| \psi g\|^2 +
    \epsilon\nu \| \psi \Delta v\|^2,
  \end{align*}
  where to control $v$ in $L^\infty$-norm we use Agmon's
    inequality and the regularity provided by
    Theorem~\ref{thm:existence-no-w}.

  Then, we estimate the terms $L_i$, $i=1, \ldots, 10$. Let us start
  with $L_1$, to get
  \begin{equation} \label{eq:L1}
    \begin{aligned}
      |L_1 | &\leq \left| \int \nabla v_t \, v \nabla
        \psi^2 \right|\\
      &\leq \left| \int v_t \nabla v \nabla \psi^2 \right| + \left|
        \int v_t v
        \Delta\psi^2 \right| \\
      & \leq C\epsilon \int | \psi \nabla v | |v_t | +
      C\epsilon^2\int |\psi v| |v_t|\\
      & \leq C\epsilon (\|\psi \nabla v\|^2 + \|v_t\|^2)
    \end{aligned}
  \end{equation}
  where we used the relation \eqref{control-on-w-funcs}.  Similarly,
  we also have that
  \begin{align*}
    |L_2
    | 
    \leq & \alpha^2 \left|\int \D_1 \nabla v_t \D_1v \nabla
      \psi^2\right| + \alpha^2 \left| \int \D_1 \nabla v_t v\D_1
      \nabla \psi^2\right|
    \\
    \leq & \alpha^2\left| \int \D_1 v_t \D_1 \nabla v \nabla
      \psi^2\right| +
    \alpha^2\left| \int \D_1 v_t \D_1 v \Delta \psi^2\right|\\
    &+ \alpha^2\left| \int \D_1 v_t \nabla v \D_1 \nabla \psi^2\right|
    + \alpha^2 \left|\int \D_1 v_t v \D_1 \Delta
      \psi^2\right| \\
    \leq & \epsilon \alpha^2C \| \D_1 v_t\|^2 + \epsilon C\alpha^2
    (\|\psi \D_1 \nabla v\|^2 + \|\psi \nabla v \|^2 ) .
  \end{align*}
  Also in this case we conclude by using \eqref{control-on-w-funcs}.
  For the term $L_3$ we have
  \begin{align*}
    |L_3|&\leq \alpha^2 \left| \int \D_1 \nabla v_t
      \nabla v \D_1\psi^2\right|  \\
    &\leq \alpha^2 \left| \int \D_1 v_t \Delta v \D_1\psi^2\right| +
    \left| \int \D_1 v_t \nabla v \D_1 \nabla\psi^2\right| \\
    &\leq \epsilon \alpha^2C \|\D_1 v_t\|^2 + \epsilon \alpha^2C (\nu
    \|\psi \Delta v\|^2 + \|\psi \nabla v\|^2).
  \end{align*}
  Next, for the terms $L_4$ and $L_5$, using the same inequalities we
  get
  \begin{align*} |L_4| \leq & 2\nu \int |\Delta v | |v | |\Delta
    \psi^2| \leq 2\nu \epsilon^2 \int |\psi\Delta v | |\psi v |
    \\
    \leq & \epsilon^2\nu \|\psi \Delta\|^2 +
    C\epsilon^2\|\psi v\|^2,
  \end{align*}
  and
  \begin{align*} |L_5|\leq &2\nu \left|\int \nabla v \Delta  v \nabla\psi^2\right| \\
    \leq & \epsilon\nu \|\psi \Delta v\|^2 + \epsilon\nu C
    \| \psi\nabla v\|^2.
  \end{align*}
  Again, for the terms $L_6$, $L_7$ and $L_8$ we have
  \begin{align*} |L_6| \leq & \nu \alpha^2 \left| \int \D_1 \nabla v
      \D_1 \Delta v \nabla \psi^2 \right| +
    \nu \alpha^2 \left|\int \D_1 \nabla v \D_1 \nabla v \Delta \psi^2 \right| \\
    \leq & \epsilon\nu\alpha^2\| \psi\D_1\Delta v\|^2 +
    C\epsilon \|\psi \D_1 \nabla v\|^2,
  \end{align*}

  \begin{align*}
    |L_7| \leq & \nu\alpha^2 \left| \int \D_1\Delta v\D_1 v \Delta
      \psi^2 \right| \leq \nu\alpha^2 \epsilon^2 \int |\psi \D_1\Delta
    v | |\psi \D_1
    v|\\
    \leq & \epsilon^2\nu\alpha^2 \|\psi \D_1 \Delta\psi \|^2
    + C\epsilon^2 \alpha^2 \nu\|\psi \nabla v\|^2
  \end{align*}
  and
  \begin{align*}
    |L_8| \leq & \nu\alpha^2 \left|\int \D_1\Delta v \Delta v \D_1
      \psi^2\right| \leq \nu\alpha^2\epsilon \int |\psi \D_1\Delta v |
    |\psi \Delta
    v| \\
    \leq &\frac{\nu \alpha^2\epsilon}{2} \|\psi \D_1\Delta v\|^2 +
    \frac{\nu\alpha^2\epsilon}{2}\|\psi\Delta v\|^2.
  \end{align*}

  Finally, for the last two terms, exploiting similar
  estimates, 
  we get
  \begin{align*}
    |L_9| \leq & \nu\alpha^2 \epsilon^2 \int |\psi \D_1 \Delta v |
    |\psi \nabla v|
    \\
    \leq & \epsilon\nu\alpha^2 \| \psi\D_1\Delta v\|^2 +
    C\epsilon^2 \psi \|\nabla v\|^2
  \end{align*}
  and
  \begin{equation} \label{eq:L10}
    \begin{aligned}
      |L_{10}| \leq & \nu\alpha^2 \int |\D_1 \Delta v| |v| |\D_1
      \Delta \psi^2| \leq \nu\alpha^3\epsilon^2 \int |\psi \D_1 \nabla
      \Delta v|
      |\psi v| \\
      \leq & \epsilon^3\nu\alpha^2\|\psi\D_1\Delta v\|^2 +
      C\epsilon^3\|\psi v\|^2.
    \end{aligned}
  \end{equation}
  Using \eqref{eq:stima} along with the estimates
  \eqref{eq:nonlinear-t}--\eqref{eq:L10} we get
  \begin{equation*} \label{eq:stima-non-chiusa}
    \begin{aligned}
      \frac{1}{2}\frac{d}{dt} \Big( &\|\psi \nabla v\|^2 + \alpha^2
      \|\psi \D_1\nabla v\|^2 \Big) + \nu \| \psi \Delta v\|^2 +
      \nu\alpha^2\|\psi
      \D_1\Delta v\|^2 \\
      \leq & \epsilon C \| v_t\|^2 + \epsilon \alpha^2C\|\D_1v_t\|^2
      + \epsilon\nu C\|\psi\Delta v\|^2\\
      & + \epsilon\nu\alpha^2 C \|\psi\D_1\Delta v\|^2 + \epsilon C\|\psi\nabla
      v\|^2 + C\|\psi g\|^2.
    \end{aligned}
  \end{equation*}
  Using the control on $\|v_t\|$ and $\|\D_1 v_t\|$ provided by
  Theorem~\ref{thm:existence-no-w} together with the Gr\"onwall
  inequality, we get the claimed regularity on $v$, i.e. $v\in
  L^\infty_{\mathrm{loc}}(0,\infty; H^{2, h}_{ \gamma})\cap
  L^2_{\mathrm{loc}}(0,\infty; H^{3, h}_{ \gamma})$. {This concludes
  STEP 1.}
\smallskip

{Before proceeding with the next steps, 
we open a parenthesis to outline the scheme behind the remaining part
of the proof.}
  Until now, we have used $v$ in place of $v^m$ for a matter of
    convenience; however, in view of extracting a convergent
    subsequence of $\{v^m\}$, here below we will employ this latter
    notation.  From the above estimates, we can extract a subsequence of $\{v^m\}$,
  still denoted by $\{v^m\}$, such that
  \begin{align*}
    &v^m \rightharpoonup \tilde v \textrm{ weak-star  in } L^\infty(\tau , T ; H^{2,h}_\gamma(\Omega)),\\
    &v^m \rightharpoonup \tilde v \textrm{ weak in } L^2(\tau , T ;
    H^{3,h}_\gamma(\Omega)).
  \end{align*}
  Moreover, as a consequence of the estimates in the proof of
  Theorem~\ref{thm:existence-no-w} we also have that
  \begin{equation} \label{strong} v^m \to \tilde v \textrm{ strong in
    } L^2(\tau , T ; H^{2,h}(\Omega)).
  \end{equation}
  To conclude our argument, {obtaining that $\{v^m\}$ is relatively
  compact in $L^2(\tau, T; H^{2, h}_\gamma(\Omega))$, we would
need some control on $d v^m/dt$.

When it is possible
to choose a special basis $w_j\in C_0^\infty(\Omega)$ to generate
 the Galerkin elements $v^m(x, t) =
  \sum_{j=1}^m a_j^m(t)w_j(x)$, $m\in \N$, such that a uniform control on
  $\|dv^m/dt\|_{L^2(\tau,T; H^{1,h}_\gamma)}(\Omega)$ holds true, this
  is enough to use a compactness result \`a la Aubin--Lions to get the
  existence of a subsequence such that $v^m \to \tilde v$ in
  $L^2(\tau, T ;H^{2,h}_\gamma(\Omega))$, and even more.

  Here, using Lemma~\ref{utility-lemma}, we obtain a similar result but not on the whole domain
  $\Omega$. Actually, what we are going to prove is the following:
  for any bounded open set $\pazocal{O} \subset \Omega$, there
  exists a subsequence of $\{v^m\}$ (depending on $\pazocal{O}$
  and relabeled $\{v^m|_{\pazocal{O}} \}$) satisfying
  \begin{equation} \label{local-c} v^m|_{\pazocal{O}} \to
    v|_{\pazocal{O}} \textrm{ in } L^2(\tau, T ;
    H^{1,h}_\gamma(\pazocal{O})).
  \end{equation}

  Since we also have that $\{v^m\}$ is weakly convergent to $\tilde v$
  in $L^2(\tau, T; H^{3,h}_\gamma)$, due to the uniqueness of the
  limit it follows that $(\tilde v)|_{\pazocal{O}}= v|_{\pazocal{O}}$
  for every bounded subset $\pazocal{O}\subset \Omega$. This fact
  along with \eqref{local-c} will be enough to
  prove that $\tilde v$ is a weak solution to \eqref{eq:Bardina-visc}
  defined in  $L^\infty_{\mathrm{loc}}(0,\infty; H^{2, h}_{ \gamma})\cap
    L^2_{\mathrm{loc}}(0,\infty; H^{3, h}_{ \gamma})$.
  Indeed, to conclude our analysis on $\Omega\times (\tau, T)$, and to
  prove that the weak formulation for $v^m$ is stable when $m \to +\infty$, we 
consider a proper family of test functions with
    separate variables and bounded supports (see, e.g., \cite{Ahn}).  Let
  $\{w_j\}_{j=1,\ldots, m}$ be the basis of the space $H_m$
  approximating $H^{3,h}_\gamma\cap H^2_0$, for $m\in \N$.  Let
  $\sigma=\sigma(t)$ be a continuously differentiable function on
  $[\tau , T ]$ with $\sigma{}(T) = 0$. Then, we set the following weak
  formulation (where $w_j(x)\sigma(t)$ are the tests) 
on 
$\Omega\X (\tau, T)$:
  \begin{align*}
    \int_\tau^T & (\psi\nabla v^m_t,\psi\nabla w_j) \sigma dt
    +\alpha^2\!\int_\tau^T(\psi\D_1\nabla v^m_t, \psi\D_1\nabla
    w_j)\sigma
    dt +\nu\! \int_\tau^T(\psi\Delta v^m, \psi\Delta w_j)\sigma dt\\
    \!\!\!  &+ \nu\alpha^2\int_\tau^T(\psi\D_1\Delta v^m,
    \psi\D_1\Delta w_j)\sigma dt + \int_\tau^T ( \psi B(v^m,v^m), \psi
    w_j \big)\sigma dt \\
    & = \int_\tau^T(\psi g,\psi w_j)\sigma dt
  \end{align*}
  for all $j = 1,\dots, m$. 
Using Lemma~\ref{utility-lemma} (the intersection $\textrm{supp}\, w_j
\cap \Omega$ is bounded) 
we will prove that the above relation passes to the limit as
  $m \to +\infty$. \smallskip

To proceed to the next steps, and prove that
  Lemma~\ref{utility-lemma} applies to our case, we set
  $X=H^{3,h}_\gamma(\pazocal O)$, $E=H^{2,h}_\gamma(\pazocal O)$,
  where $\pazocal O$ is any open set included in $\Omega$. Also,
  we choose $p=2,\, q=+\infty$ and, as already mentioned, we denote by $\Theta=(\tau,
    T)\subseteq \R$ the time interval, and by $\pazocal F = \{ v^m
    |_{\pazocal O} \}$ the approximating sequence.

\smallskip

  \noindent \textbf{STEP 2:} \emph{The approximating sequence
    $\{v^m\}$ satisfies condition \textrm{(ii)} in
    Lemma~\ref{utility-lemma}.} 
The boundedness of $v^m$ in $L^\infty(\tau, T; H^{2, h}_{ \gamma}(\pazocal{O}))\cap
  L^1(\tau, T; H^{3, h}_{ \gamma}(\pazocal{O}))$, $\pazocal{O}\subset
  \Omega$ open and bounded, follows directly from the boundedness of
  $v^m$ in $L^\infty_{\mathrm{loc}}(0,\infty; H^{2, h}_{ \gamma})\cap
  L^2_{\mathrm{loc}}(0,\infty; H^{3, h}_{ \gamma})$ proved in STEP 1.
 This concludes STEP 2.
\smallskip

\noindent
\textbf{STEP 3:} \emph{The approximating sequence $\{v^m\}$ satisfies
 condition \textrm{(i)} in Lemma~\ref{utility-lemma}.}

First, we will prove that $\{ v^m|_{\pazocal{O}} \}$ is relatively
compact in $L^2(\tau , T ; H^{2,h}_\gamma (\pazocal{O}))$ for all
bounded subsets $\pazocal{O} \subset \Omega$ by using
Lemma~\ref{utility-lemma}.  We only have to check that $\{v^m\}$
satisfies condition (i) in Lemma~\ref{utility-lemma}, i.e.,
\begin{equation*}
    \underset{k\to 0}{\lim}\, 
    \underset{m\in \N}{\sup} 
  \|\tau_k v^m -  v^m\|_{L^2(\tau ,T-k; H^{2,h}_{\gamma}(\pazocal{O}))}
  = 0.
\end{equation*}

Also here, to keep the notation as compact as
  possible, 
we write $v$ in place of $v^m$.  Consider $k > 0$ arbitrarily small
and set $V_k(t):= (v(t +k)- v(t))$. We take the product of
\eqref{eq:Bardina-visc} against $-\psi^2 w_j$, integrate in time over
$(t,t+k) \subset (\tau , T )$; subsequently, we multiply it by
\begin{equation*}
  a^m_j(t+k)-a^m_j(t)\, ,
\end{equation*}
and by summing over $j$, we reach (here below, we reintroduce the
  $dx$ in the space-depending integrals) 
\begin{equation} \label{stima-vt}
  \begin{aligned}
    \|\psi & \nabla V_k(t)\|^2 + \alpha^2
    \|\psi \D_1\nabla V_k(t)\|^2  \\
    =& - \nu \int_t^{t+k} \int \psi \Delta v\, \psi \Delta V_k(t) dx
    ds - \nu\alpha^2 \int_t^{t+k} \int \psi
    \D_1\Delta v \, \D_1\Delta V_k(t) dx ds \\
    & + \int_t^{t+k}\big( B(v,v) - g, V_k(t)\psi^2\big)ds
    -\int \nabla V_k(t)   \, V_k(t) \nabla \psi^2dx \\
    &+ \alpha^2 \int \D_1^2\nabla V_k(t) \, V_k(t)\nabla\psi^2dx -
    \alpha^2\int \D_1\nabla
    V_k(t) \nabla  V_k(t)\D_1 \psi^2 dx  \\
    &-\nu\int_t^{t+k}\int \Delta v \, V_k(t) \Delta\psi^2dxds
    - 2\nu \int_t^{t+k}\int \Delta v \, \nabla V_k(t)  \nabla\psi^2dxds \\
    & -2\nu\alpha^2\int_t^{t+k}\int \D_1 \Delta v \D_1\nabla V_k(t)
    \nabla \psi^2 \\
    &- \nu\alpha^2\int_t^{t+k}\int \D_1\Delta v \D_1 V_k(t)
    \Delta \psi^2dxds\\
    & -\nu\alpha^2\int_t^{t+k}\int \D_1\Delta v \Delta V_k(t) \D_1
    \psi^2dxds \\
    &-2\nu\alpha^2\int_t^{t+k}\int \D_1\Delta v \nabla V_k(t)
    \D_1\nabla\psi^2dxds\\
    &- \nu\alpha^2\int_t^{t+k}\int \D_1 \Delta v \, V_k(t)
    \D_1\Delta\psi^2dxds
  \end{aligned}
\end{equation}
from which, integrating on $(\tau, T- k)$ in $dt$, we get
\begin{equation}\label{eq:stima-2}
  \begin{aligned}
    \int_\tau^{T-k} &\| \psi \nabla V_k(t)\|^2dt +
    \alpha^2\int_\tau^{T-k}
    \|\psi \D_1\nabla  V_k(t)\|^2dt  \\
    =& - \nu \int_\tau^{T-k}\int_t^{t+k} \int \psi
    \Delta v\, \psi \Delta V_k(t) dx dsdt \\
    &- \nu\alpha^2 \int_\tau^{T-k} \int_t^{t+k} \int \psi
    \D_1\Delta v \, \D_1\Delta V_k(t) dx dsdt \\
    & + \int_\tau^{T-k}\int_t^{t+k}\big( B(v,v) - g, V_k(t)
    \psi^2\big)ds dt + \int_\tau^{T-k} \sum_{j=1}^{10} J_j.
  \end{aligned} \!\!\!\!\!
\end{equation}
For the terms in the right-hand side of the above equality, exploiting
the Fubini's theorem {along with the properties of the following functions}
\begin{equation*}
  \overline{s} =
  \left\{ \begin{array}{ll}
      \tau \!& \textrm{if }\, s \leq \tau\\
      s \!& \textrm{if }\, \tau < s \leq T - k\\
      T - k \!& \textrm {if }\, s > T - k\\
    \end{array} \right .\,\, \textrm{ and }\,\,\, 
  \overline{s - k} =
  \left\{ \begin{array}{ll}
      \tau \!& \textrm{if }\, s-k \leq \tau\\
      s-k \!& \textrm{if }\, \tau < s - k \leq T - k\\
      T - k \!& \textrm {if }\, s -k  > T - k\\
    \end{array}\right.
\end{equation*}
{which are used to change the order of integration, we
get}
\begin{align*} \allowdisplaybreaks \nu \bigg|\int_\tau^{T-k} &
  \int_t^{t+k} \int \psi
  \Delta v(s)\, \psi \Delta V_k(t) dx ds dt \bigg|\\
  \leq \nu &\int_\tau^{T-k} \int_t^{t+k} \| \psi
  \Delta v(s)\| \| \psi \Delta V_k(t)\| ds dt\\
  \leq \nu & \int_\tau^{T-k} \| \psi \Delta V_k(t)\| \int_t^{t+k} \|
  \psi
  \Delta v(s)\| ds dt\\
  \displaybreak[0] \leq \nu & \int_\tau^T \| \psi \Delta v(s)\|
  \int_{\overline{s-k}}^{\overline{s}} \| \psi
  \Delta  V_k(t)\| dt ds\\
  \leq \nu & \int_{t}^{T} \| \psi \Delta v(s)\|
  \Big(\int_{\overline{s-k}}^{\overline{s}}1 dt\Big)^{1/2}
  \Big(\int_{\overline{s-k}}^{\overline{s}}\| \psi \Delta  V_k(t))\|^2 dt\Big)^{1/2}ds \\
  \leq 2\nu & k^{1/2} (T-\tau)^{1/2}\| \psi \Delta v\|_{L^2(\tau, T;
    L^2)}
  \Big(  \int_{\tau}^{T}  \| \psi \Delta v(s)\|^2 ds\Big)^{1/2} \\
  \leq 2\nu & k^{1/2} (T -\tau)^{1/2} \| \psi \Delta
  v\|_{L^2(\tau, T; L^2)}^2 \\
  \leq C &\nu k^{1/2} (T -\tau)^{1/2}\,.
\end{align*}
With similar computations, we also obtain that
\begin{equation*}
  \begin{aligned}
    \nu \alpha^2\bigg| \int_\tau^{T-k} &\int_t^{t+k} \int \psi
    \D_1\Delta v \, \D_1\Delta V_k(t)  dx dsdt \bigg|\\
    &\leq 2\nu\alpha^2 k^{1/2} \| \psi \D_1\Delta
    v\|_{L^2(\tau, T; L^2)}^2 (T      -\tau)^{1/2}\\
    &\leq C\nu\alpha^2 k^{1/2} (T -\tau)^{1/2}.
  \end{aligned}
\end{equation*}  
Now, by exploiting \eqref{eq:l4-pesata}, we have that
\begin{equation*}
  \begin{aligned}
    \bigg| \int_\tau^{T-k} & \int_t^{t+k} \big(B(v,v), \psi^2 V_k(t)  \big)dsdt \bigg|\\
    \leq & \bigg|\int_\tau^{T-k}\int_t^{t+k}\int \big[\D_1 v \Delta
    V_k(t)
    \D_2 \psi^2 - \D_2 v \Delta V_k(t)   \D_1 \psi^2\big]dxdsdt \bigg|\\
    & + \bigg|\int_\tau^{T-k}\int_t^{t+k}\int \big[\D_1 v \Delta v \,
    \psi^2\D_2V_k(t)  - \D_2 v \Delta v \,\psi^2 \D_1 V_k(t) \big]dxdsdt\bigg|\\
    \leq & C\epsilon\int_\tau^{T-k}\int_t^{t+k}\|\Delta v\| \big(
    \|\nabla \D_1 v \| + \|\nabla \D_2 v \| \big) \big(\|\psi \nabla
    V_k(t)\| + \epsilon \|\psi V_k(t)\|\big)ds dt\\
    & + C\int_\tau^{T-k}\int_t^{t+k}\|\psi\Delta v\| \big( \|\nabla
    \D_1 v \| + \|\nabla \D_2 v \| \big) \big(\|\psi \Delta V_k(t)\|
    + \epsilon\|\psi \nabla V_k(t)\|)ds dt\\
    \leq & Ck^{1/2} (T-\tau)^{1/2}\big( \| v\|_{L^2(\tau, T;
      H^1_\gamma)}^2 +\| v\|_{L^2(\tau, T;
      H^{2}_\gamma)}^2) \\
    \leq & C k^{1/2} (T-\tau)^{1/2}.
  \end{aligned}
\end{equation*}

Now, we estimate the terms $J_i$, $i=1, \ldots, 10$.  Let us start
  with $J_1$ to get
\begin{equation} \label{eq:J1}
  \begin{aligned}
    \int_\tau^{T-k} |J_1 |dt &\leq \epsilon \int_\tau^{T-k} \int
    |\psi\nabla V_k(t)|\, |\psi V_k(t) |
    dxdt\\
    &\leq \epsilon\sqrt{\frac{2}{\lambda_1}} \int_\tau^{T-k} \|\psi
    \nabla V_k(t)\|^2dt.
  \end{aligned}
\end{equation}
For the terms $J_2$ and $J_3$ we have that
  \begin{equation} \label{eq:J2}
    \begin{aligned}
      \int_\tau^{T-k} |J_2|dt \leq & \alpha^2
      \epsilon\int_\tau^{T-k}\int |\psi\D_1 \nabla
      V_k(t)|\, |\psi \D_1 V_k(t) |dt\\
      &+ \alpha^2 \epsilon\int_\tau^{T-k} \int |\psi\D_1 \nabla
      V_k(t)|\, |\psi V_k(t)|dt\\
      \leq & C \alpha^2 \epsilon 
      \int_\tau^{T-k} \|\psi
      \D_1\nabla V_k(t)\|^2dt \\
      &+ C\alpha^2 \epsilon  \int_\tau^{T-k}
      \|\psi
      \D_1\nabla V_k(t)\| \|\psi  \nabla V_k(t)\| dt\\
      \leq & C\alpha^2 \epsilon
      \int_\tau^{T-k} \|\psi
      \D_1\nabla V_k(t)\|^2dt \\
      &+ C\alpha^2\epsilon (T-\tau)^{1/2} \|\psi \nabla v\|_{L^\infty
        (\tau, T; L^2)} \|\psi \nabla v\|_{L^2 (\tau, T; L^2)}
    \end{aligned}
  \end{equation}
 and that
  \begin{equation} \label{eq:J3}
    \begin{aligned}
      \int_\tau^{T-k} |J_3|dt \leq & \alpha^2 \epsilon
      \int_\tau^{T-k}\int |\psi\D_1 \nabla
      V_k(t)|\, |\psi \nabla V_k(t) |dxdt\\
      \leq & C \alpha^2 \epsilon 
      \int_\tau^{T-k} \|\psi \D_1\nabla V_k(t)\| \|\psi \nabla
      V_k(t)\|
      dt\\
      \leq & C\alpha^2 \epsilon
      \int_\tau^{T-k} \|\psi \D_1\nabla V_k(t)\| ^2  dt\\
      &+ C\alpha^2\epsilon (T-\tau)^{1/2} \|\psi \nabla v\|_{L^\infty
        (\tau, T; L^2)} \|\psi \nabla v\|_{L^2 (\tau, T; L^2)}.
    \end{aligned}
  \end{equation}
 Next, for the terms $J_4$ and $J_5$ we have
\begin{align*}
  \int_\tau^{T-k}|J_4| dt\leq & \nu\int_\tau^{T-k}\int_t^{t+k}\int
  \Delta  v  \,V_k(t)  \Delta\psi^2dxdsdt\\
  \leq & \epsilon \nu\int_\tau^{T-k}\int_t^{t+k}
  \|\Delta  v(s)\|  \, \|\psi V_k(t) \|  dsdt\\
  \leq & C\epsilon\nu
  k^{1/2}(T-\tau)^{1/2} 
  \| \psi v\|_{L^2(\tau, T; L^2)},
\end{align*}
where we used again \eqref{control-on-w-funcs}, and
\begin{align*}
  \int_\tau^{T-k} |J_5|dt\leq & 2\epsilon\nu \int_\tau^{T-k}
  \int_t^{t+k}\|\psi \Delta v(s)\| \, \|\psi \nabla V_k(t)\|
  dsdt \\
  \leq& 4\nu\epsilon k^{1/2} \| v\|_{L^2(\tau, T; H^{1}_\gamma)}
  \int_{\tau}^{T} \| \psi \Delta
  v(s)\| ds\\
  \leq & C\nu\epsilon k^{1/2} (T-\tau)^{1/2}
  \| v\|_{L^2(\tau, T; H^2_\gamma)}^2.
\end{align*}
$J_6$ is estimated as follows:
\begin{align*}
  \int_\tau^{T-k} |J_6|dt\leq & 2\nu\alpha^2\epsilon
  \int_\tau^{T-k}\int_t^{t+k}\|\psi\D_1\Delta v\|\, \|\psi\D_1  \nabla V_k(t)) \|dsdt\\
  \leq & 4\nu\epsilon k^{1/2} \alpha^2\| v\|_{L^2(\tau, T;
    H^{2,h}_\gamma)} \int_{\tau}^{T} \| \psi
  \D_1\Delta v(s)\| ds\\
  \leq & C\nu\epsilon k^{1/2} (T-\tau)^{1/2}\alpha^2\|
  v\|^2_{L^2(\tau, T; H^{3,h}_\gamma)}.
\end{align*}
In a very similar way we also get
\begin{align*}
  \int_\tau^{T-k} |J_7|dt\leq & \nu\epsilon \alpha^2C k^{1/2} \|
  v\|_{L^2(\tau, T; H^{1,h}_\gamma)}
  \int_{\tau}^{T}  \| \psi \D_1\Delta v(s)\| ds,\\
  \int_\tau^{T-k} |J_8|dt\leq & \nu\alpha^2\epsilon C k^{1/2} \|
  v\|_{L^2(\tau, T; H^{2}_\gamma) }   \int_{\tau}^{T}  \| \psi \D_1\Delta v(s)\| ds,\\
  \int_\tau^{T-k} |J_9|dt\leq & \nu\alpha^2\epsilon C k^{1/2} \|
  v\|_{L^2(\tau, T; H^{1}_\gamma)} \int_{\tau}^{T} \| \psi
  \D_1\Delta v(s)\| ds,\\
  \intertext{and} \int_\tau^{T-k} |J_{10}|dt\leq &
  \nu\alpha^2\epsilon C k^{1/2} \| v\|_{L^2(\tau, T; L^2_\gamma)}
  \int_{\tau}^{T} \| \psi \D_1\Delta v(s)\| ds.
\end{align*}
Whence, for $i=7,8, 9, 10$, we have that
\begin{equation}
  \int_\tau^{T-k} |J_i|dt\leq C k^{1/2} (T-\tau)^{1/2}
  \| v\|^2_{L^2(\tau, T; H^{3, h}_\gamma) }. 
\end{equation}
To conclude we reabsorb the terms \eqref{eq:J1}, \eqref{eq:J2} and
\eqref{eq:J3} in the left-hand side of \eqref{stima-vt}. Then, using
standard manipulations along with the above estimates we get
\begin{equation*}
  \begin{aligned}
    \int_\tau^{T-k}\!\|\psi \nabla V_k(t)\|^2dt\! +\!  \alpha^2
    \int_\tau^{T-k}\! \|\psi \D_1\nabla V_k(t)\|^2 dt \leq C
    (T-\tau)^{1/2}k^{1/2} \to 0 \textrm{ as } k \to 0;
  \end{aligned}
\end{equation*}
this concludes STEP 3.
    
    \smallskip
    
\noindent \textbf{STEP 4:} {\emph{Application of
  Lemma~\ref{utility-lemma} to $\{v^m\vert_{\pazocal{O}}\}$.}}

By Lemma~\ref{utility-lemma}, we deduce that $\{v^m|_{\pazocal{O}}\}$
is relatively compact in the space $L^2(\tau , T ; H^{2,
  h}_\gamma(\pazocal{O}))$, uniformly with respect to $\pazocal{O}$,
and we can extract a subsequence, still denoted by
$\{v^m|_{\pazocal{O}}\}$, that strongly converges in $L^2(\tau , T ;
H^{2,h}_\gamma(\pazocal{O}))\}$ for all $\pazocal{O} \subset
\Omega$. Therefore, $v^w \to v$ in
$L^2(\tau,T;(H^{2,h}_\gamma)_{\mathrm{loc}(\Omega)})$ for a suitable
$v$.

By STEP 1 we have that $\{v^m\}$ is also bounded in $L^2(\tau, T;
H^{3,h}_\gamma)$ (let us recall that $v^m|_{\pazocal{O}}$ is the
restriction of $v^m$ to $\pazocal{O}$), and hence $v^m \rightharpoonup
\tilde v$ weakly in $L^2(\tau, T; H^{3,h}_\gamma(\Omega))$.  Thus,
due to the uniqueness of the limit, one has that $(\tilde
v)|_\pazocal{O}= v$ on every ball $\pazocal{O}\subset
\Omega$. Therefore $v$ is defined on $\Omega$ and $v\in L^2(\tau,T;
H^{3,h}_\gamma(\Omega))$.  \smallskip

\noindent \textbf{STEP 5:} \emph{The limiting function $v$ is a weak
  solution.}

We now show that $v$ is a weak solution of
problem~\eqref{eq:Bardina-visc}.
%
Hence, we have to check that, for any $w_j\in H^{3,h}_\gamma \cap
H^2_0$ (the elements of the basis for the considered test functions),
the weak formulation for $v^m$ passes to the limit as $m\to +\infty$.
It is enough to verify that the nonlinear term passes to the
  limit, i.e., setting $\Omega' = \textrm{supp}\, w_j \cap \Omega$, we
  take into account the difference
\begin{align*} \allowdisplaybreaks \Bigg| \int_\tau^T & \big(B(v^m,
  v^m), \psi^2 w_j\sigma\big) dt -
  \int_\tau^T \big(B(v, v), \psi^2 w_j\big)\sigma dt\Bigg| \\
  = & \left| \int_\tau^T \big(B(v^m, v^m-v), \psi^2 w_j\big) \sigma dt
    +
    \int_\tau^T \big(B(v^m- v, v), \psi^2 w_j\sigma\big) dt\right| \\
  = & \Bigg| \int_\tau^T \int_\Omega\big[\D_1(v^m -v) \Delta v^m \D_2
  (\psi^2 w_j) - \D_2(v^m -v) \Delta v^m \D_1 (\psi^2 w_j)
  \big]\sigma dxdt\\
  &+ \int_\tau^T \int_{\Omega}\big[\D_1v \Delta (v^m-v)\D_2 (\psi^2
  w_j) - \D_2 v \Delta(v^m - v)\D_1(\psi^2 w_j)\
  \big]\sigma dxdt\Bigg| \\
  = & \Bigg| \int_\tau^T \int_{\Omega'}\big[\D_1(v^m -v) \Delta v^m
  \D_2 (\psi^2 w_j) - \D_2(v^m -v) \Delta v^m \D_1 (\psi^2 w_j)
  \big]\sigma dxdt\\
  &+ \int_\tau^T \int_{\Omega'}\big[\D_1v \Delta (v^m-v)\D_2 (\psi^2
  w_j) - \D_2 v \Delta(v^m - v)\D_1(\psi^2 w_j)\ \big]\sigma
  dxdt\Bigg| .
\end{align*}
We estimate singularly the above terms, so that
\begin{align*}
  \Bigg| &\int_\tau^T \int_{\Omega'}\big[\D_1(v^m -v) \Delta v^m
  \D_2 (\psi^2 w_j)\big]\sigma \Bigg|dxdt\\
  \leq& C\int_\tau^T \int |\D_1 (v^m -v )|\, |\psi \Delta v^m|\,
  \big( |\psi\D_2 w_j| +\epsilon |\psi w_j  |)dxdt\\
  \leq &C\int_\tau^T\|\psi \Delta v^m\|\, \| \D_1 (v^m -v )\|_{L^4}
  \big(\| \psi\D_2 w_j\|_{L^4}
  +\epsilon \|\psi w_j\|_{L^4} \big)dt\\
  \leq & C\| w_j\|_{H^{3,h}_\gamma} \int_\tau^T\|\psi \Delta v^m\|\,
  \| \D_1 \nabla (v^m -v )\|dt.
\end{align*}
Again, for the second term
\begin{align*}
  \Bigg| \int_\tau^T & \int_{\Omega'} \D_2(v^m -v) \Delta v^m \D_1
  (\psi^2 w_j)
  \big]\sigma dxdt \Bigg|\\
  \leq & C\int_\tau^T \int |\D_1 \D_2 (v^m -v )|\, |\psi \Delta v^m|\,
  |\psi
  w_j  |dxdt \\
  &+ C\int_\tau^T \int |\D_2 (v^m -v )|\, |\psi \D_1\Delta v^m|\,
  |\psi
  w_j  |dxdt\\
  \leq & C\|\psi w_j \|_{L^\infty} \Bigg[ \int_\tau^T\|\D_1\nabla
  (v^m-v)\|\, \|\psi
  \Delta v^m\|dt \\
  &+ \int_\tau^T\|\nabla (v^m-v)\|\, \|\psi \D_1\Delta v^m\|dt \Bigg]
\end{align*}
and
\begin{align*}
  \Bigg| \int_\tau^T &\int_{\Omega'}\big[\D_1v \Delta (v^m-v)\D_2
  (\psi^2 w_j)\sigma dxdt\Bigg|\\
  \leq & C\int_\tau^T \int |\psi \D_1 \nabla v |\, |\nabla (v^m- v)|
  \big(|\psi\D_2 w_j|
  +|\epsilon \psi   w_j  |\big) dxdt\\
  & + C\int_\tau^T \int |\psi \D_1 v |\, |\nabla (v^m- v)|\big(
  |\epsilon \psi w_j| + |\epsilon\psi \D_2 w_j|
  + |\epsilon \psi \nabla w_j| +  |\psi \D_2 \nabla w_j |\big) dxdt\\
  \leq & C\int_\tau^T \, \| \nabla (v^m- v)\|\, \|\psi \D_1\nabla
  v\|_{L^4}\big( \|\psi\D_2 w_j\|_{L^4}
  + \| \epsilon\psi w_j\|_{L^4} \big)dxdt\\
  & + C\int_\tau^T \|\psi \D_1 v \|_{L^\infty}\|\nabla (v^m- v)\|\big(
  \|\epsilon \psi w_j \|+ \|\epsilon\psi \D_2 w_j\| + \|\epsilon \psi
  \nabla w_j\| +
  \|\psi \D_2 \nabla w_j \|\big)dxdt\\
  \leq & C\| w_j\|_{H^{3,h}_\gamma} \int_\tau^T \| \nabla (v^m -v
  )\|\, (\|\psi \D_1 \nabla v\| + \|\psi \D_1 \Delta v\|)dt.
\end{align*}
Analogously, we have that
\begin{align*}\allowdisplaybreaks
  \Bigg| \int_\tau^T & \int_{\Omega'} \D_2 v \Delta(v^m -
  v)\D_1(\psi^2 w_j)
  \sigma dxdt\Bigg|\\
  \leq & C \int_\tau^T\int_{\Omega'} |\D_2\nabla v|\, |\nabla (v^m
  -v)|
  \,  | \D_1(\psi^2 w_j)| dxdt \\
  \displaybreak[0] &+ C \int_\tau^T\int_{\Omega'} |\D_2 v|\, |\nabla
  (v^m -v)|
  \, |\D_1\nabla(\psi^2 w_j)| dxdt\\
  \leq & C \|w_j\|_{H^{3,h}_\gamma} \int_\tau^T \|\nabla (v^m -v)\|\,
  \|\psi\D_2\nabla v\|dt\\
  &+ \int_\tau^T \|\nabla (v^m -v)\|\, \|\psi\D_2 v\|_{L^4} \|\big(
  \|\epsilon \psi w_j\|_{L^4} + \|\epsilon\psi \D_1 w_j\|_{L^4}\\
  &+ \|\epsilon \psi \nabla w_j\|_{L^4} +  \|\psi \D_1 \nabla w_j \||_{L^4}\big) dt\\
  \leq & C\|w_j\|_{H^{3,h}_\gamma} \int_\tau^T \|\nabla (v^m -v)\|\,
  \|\psi\D_2\nabla v\|dt\,.
\end{align*}
Therefore, all the above four terms go to $0$ as $m$ goes to
$+\infty$. \smallskip

\noindent{\textbf{STEP 6:}} \emph{Continuity, i.e. $v\in C([0, T],
  H^{1,h}_\gamma)$, $T>0$.}

Fix any $T>0$. First, we notice that $v_t, \D_1 v_t \in
L^2(0,T;H^1_\gamma(\Omega)')$ since, thanks to
Theorem~\ref{thm:existence-no-w}, they belong to
$L^2(0,T;H^1(\Omega))$ and $H^1_\gamma\subset H^1 \subset
(H^1_\gamma)'$. Moreover, from Theorem~\ref{thm:weight-existence}, we
have in particular that $v, \D_1 v \in
L^2(0,T;H^1_\gamma(\Omega))$. By interpolation, we conclude that $v,
\D_1 v \in C([0,T];L^2_\gamma(\Omega))$, which is the claim.}

\end{proof}

\appendix

\section{Properties of the weight function} \label{app:weight}

First, we introduce the function $\phi (x, \epsilon, \gamma)=\phi (x,
3, 2, \epsilon, \gamma) = (1 + |\epsilon x_1|^{3}+|\epsilon
x_2|^{2})^{\gamma}$ and observe that
\begin{equation} \label{control-on-derivatives} |\D^{\beta} \phi| \leq
  C \epsilon^{|\beta |} \phi^{1/2},
\end{equation}
for every multi-index $\beta=(\beta_1,\beta_2)$, $0<|\beta|\leq 3$ and
$\beta_2\leq 2$, provided $0<\gamma\leq 2/3$.

Actually, $|\D_1 \phi| \leq 3 \epsilon \phi^{1/2}$ if and only if
\begin{equation*}
  |\D_1 (1 + |\epsilon x_1|^{3}+
  |\epsilon x_2|^{2})^{\gamma}|
  = \frac{3\gamma \epsilon |\epsilon x_1|^2}{(1 +
    |\epsilon x_1|^{3}+|\epsilon x_2|^{2})^{1-\gamma} }
  \leq 3 \epsilon   (1 + |\epsilon x_1|^{3}+
  |\epsilon x_2|^{2})^{\gamma/2},
\end{equation*}
which is satified provided that
\begin{equation} \label{derivative} 2\leq 3\left(1
    -\frac{1}{2}\gamma\right) \iff 0 < \gamma \leq \frac{2}{3}.
\end{equation}
Clearly, under this hypothesis we also have that $|\D_2 \phi| \leq
3\epsilon \phi^{1/2}$.
 
Moreover, we have that the following relation holds true:
\begin{equation*}
  \begin{aligned}
    |\D_1^2 (1 + |\epsilon x_1|^{3}+ |\epsilon x_2|^{2})^{\gamma}|
    &\leq \frac{6\epsilon^2|\epsilon x_1|}{(1 + |\epsilon
      x_1|^{3}+|\epsilon x_2|^{2})^{1-\gamma}} +
    \frac{9(1-\gamma)\epsilon^2|\epsilon x_1|^4}{(1 +
      |\epsilon x_1|^{3}+|\epsilon x_2|^{2})^{2-\gamma}}\\
    & \leq 15 \epsilon^2 (1 + |\epsilon x_1|^{3}+ |\epsilon
    x_2|^{2})^{\gamma/2},
  \end{aligned}
\end{equation*}
provided that $\gamma \leq 4/3$. Under this last condition one can
easily check that $|\D^{\beta} \phi| \leq C\epsilon^{|\beta|}
\phi^{1/2}$ for every multi-index $\beta=(\beta_1,\beta_2)$,
$|\beta|\leq 2$.  More in general, taking $\gamma \leq 2/3$, as in
\eqref{derivative}, we have \eqref{control-on-derivatives}.

\medskip

Now, consider the map $g(\tau)$, $\tau\geq 0$ (see \cite[(3.5),
p. 561]{Babin-1992}), given by
\begin{equation} \label{function-g}
  \begin{aligned}
    & g(\tau) = 1/4 + \tau^2,  \quad \quad && 0\leq \tau\leq 1/2\\
    &g(\tau) = \tau ,  \quad &&  1/2\leq \tau\leq \rho\\
    &g(\tau) = \rho + 1/2 - (\rho + 1-\tau)^2/2 , \quad && \rho\leq
    \tau\leq \rho +1\\
    &g(\tau) = \rho + 1/2,\quad &&\quad \! \quad\tau\geq \rho+1,
  \end{aligned}
\end{equation}
and obviously $g(\tau)= \tau$ when $\rho = +\infty$, $\tau \geq 1/2$.
 
Define the weight function $\varphi$ as
\begin{equation} \label{weight-funcs} \varphi(x_1, x_2, \epsilon,
  \rho, \gamma):= \Big(g\big((1+|\epsilon x_1|^3 + |\epsilon
  x_2|^2)^{1/2}\big)\Big)^{2\gamma};
\end{equation}
then it holds true that
\begin{equation*}
  \underset {\rho \to +\infty}{\lim} \varphi(x_1, x_2, \epsilon, \rho,
  \gamma)
  = (1+|\epsilon x_1|^3 +
  |\epsilon x_2|^2)^{\gamma}.
\end{equation*}

We are ready to show Lemma~\ref{w-funcs}.

\begin{proof}[Sketch of the proof of Lemma~\ref{w-funcs}]
  This is due, essentially, to the fact that $0\leq g'\leq 1$,
  $g'(\tau)\equiv 0$ when $\tau > \rho+1$, to $g(\tau) \sim \tau$ when
  $1/2\leq \tau\leq \rho+1$, and to the properties of the weight
  $\phi$ when $\gamma \leq 2/3$.

  Take
  \[
  \tau= \phi (x, 3, 2, \epsilon, 1/2) =(1 + |\epsilon x_1 |^3 +
  |\epsilon x_2|^2)^{1/2}
  \]
  and begin by considering $|\beta|=1$, with $\D^\beta =\D_1$ (the
  case of $\D^\beta =\D_2$ is easier and left to the reader). We want
  to prove that
  \begin{equation} \label{eq:first-derivative}
    {|\D_1 \varphi| = 2\gamma g(\phi)^{2\gamma-1}
      g'(\phi) |\D_1 \phi|} \leq C\epsilon \varphi^{1/2},
  \end{equation}
  where $\phi := \phi(x, 3, 2, 1/2)$.

  From ${|\D_1 \phi|=3\epsilon} |\epsilon x_1|^2/2\phi \leq
  C{\epsilon}|\epsilon x_1|^{1/2}$ and the definition
  \eqref{function-g}, we get {\begin{equation*} |\D_1 \varphi| \leq
      \left\{ \begin{array}{ll}
          \gamma \epsilon C(1/4 + \phi^2)^{2\gamma -1} |\epsilon x_1|^{1/2},  &  0\leq \phi\leq 1/2\\
          \gamma \epsilon C\phi^{2\gamma -1} |\epsilon x_1|^{1/2}, &
          1/2\leq \phi\leq
          \rho\\
          \gamma \epsilon C \Big(\rho + 1/2 - (\rho +1 -
          \phi)^2/2\Big)^{2\gamma -1} (\rho + 1 - \phi) |\epsilon
          x_1|^{1/2}, & \rho\leq \phi\leq
          \rho + 1\\
          0 & \qquad\! \phi\geq \rho + 1.
        \end{array}
      \right.
    \end{equation*}}
  Consider the case of $\D_1 \varphi$ when $\rho \leq \phi \leq
  \rho+1$, the others are similar.

  The condition ${|\D_1 \varphi|} \leq \epsilon C \varphi^{1/2}$
  {yields if we show that}
  \begin{equation*}
    {\gamma \epsilon}  \Big(\rho + 1/2 - (\rho +1 -
    \phi)^2/2\Big)^{2\gamma -1} (\rho + 1 - \phi) |\epsilon x_1|^{1/2}
    \leq \epsilon C \Big(\rho + 1/2 - (\rho +1 -
    \phi)^2/2\Big)^{\gamma},
  \end{equation*}
  i.e.
  \begin{equation*}
    \gamma \Big(\rho + 1/2 - (\rho +1 -
    \phi)^2/2\Big)^{\gamma -1} (\rho + 1 - \phi) |\epsilon x_1|^{1/2}\leq C.
  \end{equation*}
  Hence, recalling that $\rho \leq \phi\leq \rho +1$, we obtain
  \begin{equation} \label{eq:utility-computation-1}
    \begin{aligned}
      \frac{ (\rho + 1 - \phi) |\epsilon x_1|^{1/2}}{\Big(\rho + 1/2 -
        (\rho +1 - \phi)^2/2\Big)^{1- \gamma } } &\leq \frac{
        |\epsilon x_1|^{1/2}}{\Big(\rho + 1/2 - (\rho +1 -
        \phi)^2/2\Big)^{1- \gamma } } \\
      &\leq \frac{ |\epsilon x_1|^{1/2}}{\rho^{1- \gamma } } \leq
      \frac{\phi^{1/3}}{\rho^{1-\gamma}} \leq \frac{(1 +
        \rho)^{1/3}}{\rho^{1-\gamma}}
    \end{aligned}
  \end{equation}
  which is bounded for $\gamma \leq 2/3$. Then, relation
  \eqref{eq:first-derivative} follows.

  Now, consider the case of the second derivatives (actually we take
  into account just the mixed partial derivatives $\D^\beta\varphi =
  \D^2_{12}\varphi$ ).  We have that
  \begin{align*}
    \D_2 \Big( 2\gamma g(\phi)^{2\gamma-1} g'(\phi) \D_1 \phi \Big) =
    & 2\gamma (2\gamma -1) g(\phi)^{2\gamma-2}(g'(\phi))^2\partial_2
    \phi \partial_1 \phi \\
    & + 2\gamma g(\phi)^{2\gamma -1 }g''(\phi) \D_2\phi\D_1\phi +
    2\gamma g(\phi)^{2\gamma -1 }g'(\phi) \D_{12}^2\phi\\
    =: & A_1 + A_2 + A_3.
  \end{align*}
  Let $\rho \leq \phi\leq \rho+1$. In what follows we develop the
  calculations only for this case; the others are similar, if not more
  elementary.

  We prove that $|A_i|\leq C\epsilon^2 \varphi^{1/2}$, for
  $i=1,2,3$. First consider $A_1$. {Recalling that $x_2$ is bounded,
    we} have that
  \begin{equation*}
    {|A_1|}\leq C\epsilon^2 g(\phi)^{2\gamma-2} \big( \rho +1 -\phi\big)^2
    |\epsilon x_1|^{1/2} \leq C\epsilon^2 g(\phi)^{2\gamma-2}  |\epsilon x_1|^{1/2},
  \end{equation*}
  and hence condition $|A_1|\leq C\epsilon^2 \varphi^{1/2}$ follows by
  requiring
  \begin{equation}
    \epsilon^2 g(\phi)^{2\gamma-2}  |\epsilon x_1|^{1/2} \leq C\epsilon^2
    g(\phi)^\gamma \textrm{ or equivalently }
    g(\phi)^{\gamma-2}  |\epsilon x_1|^{1/2} \leq C,
  \end{equation}
  with $\rho\leq \phi \leq \rho +1$.  Exploiting the same computations
  as in \eqref{eq:utility-computation-1}, we obtain
  \begin{equation*}
    \frac{ |\epsilon x_1|^{1/2}}{\Big(\rho + 1/2 - (\rho +1 -
      \phi)^2/2\Big)^{2- \gamma } } \\
    \leq \frac{ |\epsilon x_1|^{1/2}}{\rho^{2- \gamma } }
    \leq \frac{\phi^{1/3}}{\rho^{2-\gamma}}
    \leq \frac{(1 + \rho)^{1/3}}{\rho^{2-\gamma}} \leq C,
  \end{equation*}
  where the last inequality is satisfied for $\gamma\leq 5/3$.

  As for $A_2$, one can see that $|A_2|\leq C\epsilon^2
  g(\phi)^{2\gamma-1} |\epsilon x_1|^{1/2}$.  So $|A_2|\leq
  C\epsilon^2 \varphi^{1/2}$ if
  \begin{equation*}
    \epsilon^2 g(\phi)^{2\gamma-1} |\epsilon x_1|^{1/2} \leq C\epsilon^2
    g(\phi)^{\gamma}
    \textrm{ or equivalently } g(\phi)^{\gamma-1} |\epsilon x_1|^{1/2}
    \leq C,
  \end{equation*}
  which is certainly verified under the assumption that $\gamma \leq
  2/3$.

  {Since $|A_3|\leq C g(\phi)^{2\gamma -1} |\D_{12}^2\phi|$, the
    condition $|A_3|\leq C\epsilon^2 \varphi^{1/2}$ is satisfied
    provided that
    \begin{equation} \label{conticonti-e-conti} \epsilon^2
      g(\phi)^{2\gamma -1} \frac{|\epsilon x_1|^2 |\epsilon
        x_2|}{\phi^3} \leq C\epsilon^2 g(\phi)^\gamma, \textrm{ that
        is } g(\phi)^{\gamma -1} \frac{|\epsilon x_1|^2 |\epsilon
        x_2|}{\phi^3} \leq C.
    \end{equation}
    We conclude using the fact that $\gamma\leq 2/3$ and that
    \[
    \frac{|\epsilon x_1|^2 |\epsilon x_2|}{\phi^3}\leq
    \frac{\phi^{7/3} }{\phi^3} \leq 1\,.
    \]}

  Using the previous computations along with the fact that
  $g'''(\tau)\equiv 0$, the general case \eqref{control-on-w-funcs}
  follows directly.
\end{proof}

\noindent \textbf{Acknowledgments:} The authors are members of
  INdAM and GNAMPA. The first author was partially supported by the
  project GNAMPA 2015 ``{\it Dinamiche non autonome, sistemi
    hamiltoniani e teoria del controllo}''.

\end{document}